\documentclass[11pt]{article}

\usepackage[margin=1in]{geometry}
\usepackage{amsmath,amssymb,amsthm,mathtools}
\usepackage{enumitem}
\usepackage[hidelinks]{hyperref}

\newtheorem{theorem}{Theorem}[section]
\newtheorem{proposition}[theorem]{Proposition}
\newtheorem{lemma}[theorem]{Lemma}

\newtheorem{remark}[theorem]{Remark}
\theoremstyle{definition}
\newtheorem{example}[theorem]{Example}
\theoremstyle{plain}

\DeclareMathOperator{\gph}{gph}
\DeclareMathOperator{\dom}{dom}
\DeclareMathOperator{\ri}{ri}

\title{Maximal monotonicity of piecewise polyhedral mappings}
\author{
Xudong Li\thanks{School of Data Science, Fudan University, Shanghai, China. Email: \texttt{lixudong@fudan.edu.cn}.}
\and
R. Tyrrell Rockafellar\thanks{Department of Mathematics, University of
Washington, Seattle, WA 98195-4350, USA. Email: \texttt{rtr@uw.edu}.}
\and
Defeng Sun\thanks{Department of Applied Mathematics, The Hong Kong Polytechnic
University, Hung Hom, Hong Kong. The research of Defeng Sun was supported in part by the Hong Kong Research Grants Council under GRF Project No. 15309625. Email: \texttt{defeng.sun@polyu.edu.hk}.}
}
\date{August 26, 2026}

\begin{document}
\maketitle

\begin{abstract}
Maximal monotone mappings that are piecewise polyhedral arise from
subdifferentials of convex functions and saddle functions that are
piecewise linear-quadratic and enter into algorithmic constructions of
importance in linear-quadratic optimization and associated splitting
methods.  The question of whether those constructions preserve maximal
monotonicity is then crucial.  The usual answers to that invoke constraint
qualifications involving the nonemptiness of intersections of certain
relative interiors, but it is shown here that the piecewise polyhedral
structure allows the relative interiors to be bypassed.
\end{abstract}

\noindent\textbf{Keywords.}
Maximal monotone mappings; piecewise polyhedral multifunctions; resolvents;\\
Minty's theorem; nonexpansive mappings.

\section{Introduction}
Set-valued mappings that are maximal monotone are ubiquitous in convex
analysis and its applications to optimization and equilibrium modeling, in
particular the design and validation of solution algorithms.  Prime examples
are the subdifferential mapping associated with a closed proper convex
function and, as a special case, the normal cone mapping associated with a
nonempty closed convex set, but maximal monotonicity also arises in
connection with saddle functions and in other ways.  The fundamentals in
finite dimensions are covered in 
\cite[Chapter 12]{RockafellarWets1998} and in infinite dimensions in
\cite{BauschkeCombettes2017}.  An important issue in the
theory is how to know whether a construction, such as the addition of two
or more mappings, or some kind of composition, will preserve maximal
monotonicity.  This has been very much studied --- but there is nonetheless
a crucial gap in the finite-dimensional setting, which will be filled here.

The gap has to do with the ``constraint qualifications'' usually needed to
ensure that the monotonicity retains its maximality.  These resemble the
constraint qualifications in Lagrange multiplier theory and, in convex
optimization, involve nonempty intersections of interiors or relative
interiors of certain convex sets.   It is well known, however, that
constraint qualifications are not needed for linear constraints.  The
conjecture is that they ought not to be needed for maximal monotonicity
when dealing with the mappings associated with piecewise linear-quadratic
models \cite[Definition 10.20]{RockafellarWets1998} in optimization and equilibrium, and that is crucial because such
mappings are more prominent than ever in data science today.

The key is focusing on set-valued mappings that are {\it piecewise
polyhedral\/} in having their graphs be the union of finitely many
polyhedral convex sets.  This is true for the subdifferential of a
piecewise linear-quadratic convex function and for the normal cone mapping
associated with a polyhedral convex set, as well as for relations arising
in dual or primal-dual proximal-point schemes \cite{Rockafellar1976PPA}.
See also, for example, \cite[Propositions~12.29--12.30
and Example~12.31]{RockafellarWets1998}.  Our question about preservation
of maximality under various constructions is, in more precise terms, the
following:
\vskip.2cm
\begin{center}
	\noindent \parbox{0.92\textwidth}{
	{\it When the underlying maximal monotone mappings are piecewise polyhedral,
           can the classical relative-interior assumptions be
           replaced by plain nonempty intersection assumptions?}
	}
\end{center}
\vskip.2cm
The constructions in view go back to the convex analysis of subdifferentials
and normal cones~\cite{Rockafellar1970} and concern linear composition,
addition and argument restriction.  The existing results that utilize
relative interiors in constraint qualifications are laid out in
\cite[Theorem~12.43, Corollary~12.44, and Exercise~12.46]{RockafellarWets1998}.
The new results here confirm a positive answer to our question of
simplification.

Our confirmation centers on proving a projection theorem.  Let
\(H\) be a finite-dimensional real Hilbert space with an orthogonal
decomposition
\(
H=H_*\oplus H_{**},
\)
and let \(T:H\rightrightarrows H\) be maximal monotone with piecewise
polyhedral graph. Define \(T_*:H_*\rightrightarrows H_*\) by
\begin{equation}\label{eq:projected}
	w_*\in T_*(z_*)
	\quad\Longleftrightarrow\quad
	\exists \, z_{**}\in H_{**}\ \hbox{such that}\ (w_*,0)\in T(z_*,z_{**}).
\end{equation}
We show that if
\begin{equation}\label{eq:slice-nonempty}
	\exists \, (z_*,z_{**}, w_*)\in H_*\times H_{**}\times H_*
	\quad\hbox{such that}\quad
	(w_*,0)\in T(z_*,z_{**}),
\end{equation}
then \(T_*\) is maximal monotone.

The proof technique is resolvent-based.  It will rely on the fact that,
by Minty's parametrization \cite{Minty1962}, maximal monotonicity is
equivalent to the resolvent domain being the whole ambient space.  Since
\(T\) is maximal monotone and \(\gph T\) is piecewise polyhedral, its
resolvent
\(
  (I+T)^{-1}
\)
is a single-valued nonexpansive mapping that is piecewise affine on
$H$.  We will demonstrate that the resolvent of \(T_*\) can in this case be
described through the fixed-point condition in the eliminated component.
Then, the key step is a global fixed-point lemma for nonexpansive
piecewise affine mappings:  if this fixed-point set is nonempty, its
projection onto \(H_*\) has to be all of \(H_*\).  This directly yields the
full Minty domain and thus maximality for \(T_*\).

The projection theorem is the basic engine that powers all the other
preservation rules.  Argument restriction is handled by applying the
projection result to an inverse, while linear composition is reduced to
restriction on the range of the linear map.  The addition theorem then
follows by applying the composition result to the product mapping along
the diagonal.  Therefore, once the projection theorem is established, the
piecewise polyhedral versions of the standard preservation results follow
systematically.

The remainder of the paper is organized as follows. Section \ref{sec:preliminaries} fixes the
notation and collects the preliminary facts on projected graph slices and
resolvents. Section \ref{sec:projection} establishes the projection theorem, our main
technical result. Section \ref{sec:preservation} derives from it the preservation rules for
argument restriction, linear pullbacks, and addition.
In Section \ref{sec:conclusion}, we conclude the paper.

\section{Preliminaries}\label{sec:preliminaries}

We collect here the notation and elementary facts used in the
sequel.   
Let \(M\) be a set-valued mapping. Its graph,
domain, and range are denoted by
\[
\gph M:=\{(x,u)\mid u\in M(x)\}, \  
\dom M:=\{x\mid M(x)\neq\emptyset\}, \ \operatorname{range} M:=\{u \mid \exists\, x \mbox{ with } u \in M(x) \}.
\]
{Its inverse mapping \(M^{-1}\) is defined by
	\(
	M^{-1}(y) := \{x \mid y \in M(x)\}
	\) with $\dom M^{-1}=\operatorname{range}M$.}
Let $H$ be a finite-dimensional real Hilbert space with inner product \(\langle\cdot,\cdot\rangle\) and its induced norm \(\|\cdot\|\). 
A mapping \(F: H \rightrightarrows H\) is monotone if
\[
\langle x-y,u-v\rangle\geq 0
\quad
\text{for all }(x,u),(y,v)\in\gph F,
\]
and is maximal monotone if it is monotone and $\gph F$ is not properly
contained in the graph of any other monotone mapping on $H$. 
Throughout the paper, an orthogonal decomposition is written
as
\(
H=H_*\oplus H_{**}.
\)
The associated orthogonal projections are denoted by \(P_*\) and \(P_{**}\),
and the identity mappings on \(H\), \(H_*\), and \(H_{**}\) are denoted by
\(I\), \(I_*\), and \(I_{**}\), respectively.  The same convention will be used
with other ambient spaces when they are introduced locally in a statement.
If \(V\subset H\) is a subspace, then
\(
J_V:V\hookrightarrow H
\)
denotes the canonical inclusion, that is, \(J_Vv=v\) for every \(v\in V\).  Its adjoint is the orthogonal projection onto \(V\):
\(
J_V^*=P_V:H\to V.
\)

A subset of
a finite-dimensional real Hilbert space is called piecewise polyhedral if it is a
finite union of convex polyhedra.  A set-valued mapping is called piecewise
polyhedral when its graph is piecewise polyhedral.  For foundational continuity properties of piecewise polyhedral mappings, see \cite{Robinson1981Polyhedral} and \cite[Example 9.57]{RockafellarWets1998}. A single-valued mapping is
called piecewise affine if its domain can be covered by finitely many
polyhedra on each of which the mapping is affine.
We will repeatedly use the following elementary preservation properties of piecewise polyhedral sets.

\begin{lemma}[Elementary polyhedral calculus]\label{lem:pp-calculus}
	Finite intersections, Cartesian products, and linear images and inverse images of piecewise polyhedral sets are piecewise polyhedral.
\end{lemma}

\begin{proof}
	This follows by writing the sets as finite unions of convex polyhedra.
	The assertion for linear images and inverse images follows from
	\cite[Theorem~19.3]{Rockafellar1970}. The other assertions are immediate
	from the corresponding polyhedral representations.
\end{proof}

Next, we record a simple graph-projection fact.  It will be applied in the
main theorem to the projected relation \(T_*\), but the statement is written
in a general form for later reference.

\begin{proposition}[Projected graph slices]\label{prop:basic}
Let \(Z=Z_*\oplus Z_{**}\) be a finite-dimensional real Hilbert space, and let
\(M:Z\rightrightarrows Z\) be monotone and piecewise polyhedral.  Define
\(M_*:Z_*\rightrightarrows Z_*\) by
\[
  v_*\in M_*(z_*)
  \quad\Longleftrightarrow\quad
  \exists\, z_{**}\in Z_{**}\ \text{such that}\ (v_*,0)\in M(z_*,z_{**}).
\]
Then \(M_*\) is monotone and piecewise polyhedral.
\end{proposition}

\begin{proof}
Let \((z_i,v_i)\in\gph M_*\), \(i=1,2\).  By definition, there are
\(z_{i,**}\in Z_{**}\) such that
\(
  (v_i,0)\in M(z_i,z_{i,**}).
\)
The monotonicity of \(M\) gives
\[
  \big\langle (z_1,z_{1,**})-(z_2,z_{2,**}),
              (v_1,0)-(v_2,0)\big\rangle\ge 0,
\]
and hence
\(
  \langle z_1-z_2,v_1-v_2\rangle\ge 0.
\)
Thus \(M_*\) is monotone. Moreover,
\[
  \gph M_*
  =\big\{(z_*,v_*):\exists\, z_{**}\in Z_{**}
       \text{ such that } (z_*,z_{**},v_*,0)\in\gph M\big\},
\]
which is the linear projection of the intersection of \(\gph M\) with the
subspace \(Z_*\oplus Z_{**}\oplus Z_*\oplus\{0\}\).  By Lemma \ref{lem:pp-calculus}, \(\gph M_*\), or equivalently \(M_*\), is piecewise polyhedral. 
\end{proof}

The next two facts are standard consequences of Minty's parametrization. Specifically, Minty's criterion
detects maximality, while the polyhedral resolvent lemma supplies the
piecewise affine structure of resolvents.  

\begin{lemma}[Minty resolvent criterion]\label{lem:resolvent-criterion}
Let \(M:Z\rightrightarrows Z\) be monotone on a finite-dimensional real Hilbert
space \(Z\).  Then, on its domain, the resolvent
\((I+M)^{-1}\) is single-valued and nonexpansive.
Moreover,
\(M\) is maximal monotone if and only if
\[
\dom(I+M)^{-1}=Z.
\]
\end{lemma}

\begin{proof}
This is Minty's theorem in finite dimensions. See \cite{Minty1962} and
\cite[Theorem~12.12]{RockafellarWets1998}.
\end{proof}

\begin{lemma}[Piecewise affine resolvent]\label{lem:closed-domain}
Let \(M:Z\rightrightarrows Z\) be monotone and piecewise polyhedral on a
finite-dimensional real Hilbert space \(Z\). Then \((I+M)^{-1}\) is piecewise affine on its domain.
\end{lemma}

\begin{proof}
Define the linear operator 
\(
  \Phi:Z\times Z\to Z\times Z\) by \(
  \Phi(x,u):=(x+u,x).
\)
Then
\[
  \gph (I + M)^{-1}
  =\{(a,x)\mid a-x\in M(x)\}
  =\{(x+u,x)\mid (x,u)\in\gph M\}
  =\Phi(\gph M),
\]
which, together with Lemma \ref{lem:pp-calculus}, implies that \(\gph (I+M)^{-1}\), or equivalently $(I + M)^{-1}$, is piecewise
polyhedral.
By Lemma~\ref{lem:resolvent-criterion}, \((I + M)^{-1}\) is single-valued on \(
D_M:=\dom(I+M)^{-1}.
\)
It therefore follows from
\cite[Exercise~2.48]{RockafellarWets1998}
that \((I + M)^{-1}\) is piecewise affine on \(D_M\). 
\end{proof}

For a maximal monotone piecewise polyhedral mapping \(M\),
the piecewise-affine resolvent property was also established in
\cite[Theorem~3.4]{BillupsFerris1999}.

\section{The Projection Theorem}\label{sec:projection}

We now return to the projection operation introduced in \eqref{eq:projected}.
Throughout this section we fix the orthogonal decomposition
\(
  H=H_*\oplus H_{**},
\)
and use the notation introduced in Section~\ref{sec:preliminaries}.  The
mapping \(T:H\rightrightarrows H\) is assumed to be maximal monotone and
piecewise polyhedral, and \(T_*:H_*\rightrightarrows H_*\) is defined by
\eqref{eq:projected}.  We also assume the nonempty slice condition
\eqref{eq:slice-nonempty}.

By Proposition~\ref{prop:basic}, the projected relation \(T_*\) is monotone
and piecewise polyhedral.  Thus maximality is the only remaining issue.  We
prove it by passing to Minty coordinates.  The maximality of \(T\) gives the
nonexpansive resolvent \cite[Theorem~12.12]{RockafellarWets1998}
\[
  R:=(I+T)^{-1}:H\to H.
\]
Then Lemma~\ref{lem:closed-domain} implies that $R$ is
piecewise affine on \(H\).  The next lemma identifies the resolvent domain of
\(T_*\) as a projected fixed-point set for \(R\).

\begin{lemma}[Projected resolvent representation]
\label{lem:projected-resolvent}
Let
\(
 R_*:=(I_*+T_*)^{-1}.
\)
For \(a_*,x_*\in H_*\), 
$x_*=R_*a_*$ if and only if there exists \(b_{**}\in H_{**}\) such that $R(a_*,b_{**})=(x_*,b_{**}).$
Consequently,
\begin{equation}\label{eq:D-star-fixed-point}
  \dom R_*
  =P_*
  \Big\{
    (a_*,b_{**})\in H_*\oplus H_{**}
    \;\Big|\;
    P_{**}R(a_*,b_{**})=b_{**}
  \Big\}.
\end{equation}
\end{lemma}

\begin{proof}
The condition \(x_*=R_*a_*\) means
\(
  a_*-x_*\in T_*(x_*).
\)
By the definition of \(T_*\), this is equivalent to the existence of
\(b_{**}\in H_{**}\) such that
\(
  (a_*-x_*,0)\in T(x_*,b_{**}).
\)
Adding \((x_*,b_{**})\) gives
\(
  (a_*,b_{**})\in (I+T)(x_*,b_{**}),
\)
which is equivalent to
\[
  R(a_*,b_{**})=(x_*,b_{**}).
\]
The converse follows by reversing these implications.  Eliminating \(x_*\)
gives \eqref{eq:D-star-fixed-point}.
\end{proof}

From Lemmas \ref{lem:resolvent-criterion} and \ref{lem:projected-resolvent}, we see that to show the maximality of $T_*$, it suffices to show that for every \(a_*\in H_*\), there exists \(b_{**}\in H_{**}\) such that
\[
P_{**}R(a_*,b_{**})=b_{**}.
\]
Thus, the maximality of \(T_*\) is reduced to a fixed-point statement for nonexpansive piecewise affine mappings. This is precisely where the piecewise polyhedral structure enters the argument. 
For fixed \(a_*\), piecewise
affineness makes the range of the displacement mapping
\(
I_{**}-P_{**}R(a_*,\cdot)
\)
a finite union of polyhedra and hence closed. The lemma below shows
that this displacement mapping is maximal monotone, so its range is
nearly convex. These properties allow us to invoke a separation argument, leading to the desired fixed-point claim.

\begin{lemma}[Global projected fixed-point lemma]
\label{lem:global-projected-fixed-point}
Let \(X\) and \(Y\) be finite-dimensional real Hilbert spaces, and let
\(\Omega:X\oplus Y\to X\oplus Y\) be nonexpansive and piecewise affine with $\dom \Omega = X\oplus Y$.
Define
\[
  K:=\{(a,b)\in X\oplus Y\mid P_Y\Omega(a,b)=b\}.
\]
If \(K\ne\emptyset\), then
\[
  P_XK=X.
\]
\end{lemma}

\begin{proof}
Choose \((\bar a,\bar b)\in K\).  Thus
\(
  P_Y\Omega(\bar a,\bar b)=\bar b.
\)
Fix an arbitrary \(a\in X\).  We prove that there exists \(b\in Y\) such that
\(P_Y\Omega(a,b)=b\).  Define $Q_a:Y\to Y$ by $Q_a(b):=P_Y\Omega(a,b)$, and set
\(
  G_a:=I_Y-Q_a.
\)
Then, $\dom \Omega = X\oplus Y$ implies that $\dom Q_a = \dom G_a = Y$.

Since \(P_Y\) and \(\Omega\) are nonexpansive, \(Q_a\) is
nonexpansive. For any 
\(b_1,b_2\in Y\), 
we have
\[
\begin{aligned}
&2\langle G_a(b_1)-G_a(b_2),b_1-b_2\rangle
 -\|G_a(b_1)-G_a(b_2)\|^2\\
={}&\|b_1-b_2\|^2-\|Q_a(b_1)-Q_a(b_2)\|^2
\ge 0.
\end{aligned}
\]
Hence, $G_a$ is continuous and monotone on \(\operatorname{dom}G_a=Y\). Then, by \cite[Example 12.7]{RockafellarWets1998}, \(G_a\) is maximal monotone on \(Y\). 
Moreover, since \(\Omega\) is piecewise affine, Lemma \ref{lem:pp-calculus},  \cite[Exercise~2.48]{RockafellarWets1998} and the definition of \(G_a\) imply that \(G_a\) is also piecewise affine.
Thus
\(
  C:=\operatorname{range} G_a
\)
is a finite union of polyhedra and therefore is closed.  By \cite{Rockafellar1970Virtual} and \cite[Theorem~12.41]{RockafellarWets1998}, 
the maximal monotonicity of \(G_a\) implies that \(C\) is
nearly convex, in the sense that there exists a convex set \(C_0\subset Y\) such that
\(
  C_0\subset C\subset \operatorname{cl} C_0.
\)
The closedness of \(C\) then implies that 
\(
  C=\operatorname{cl} C_0.
\)
Since $C_0$ is convex, we have from \cite[Theorem 6.2]{Rockafellar1970} that \(C\) is also convex.

Suppose, for contradiction, that \(Q_a\) has no fixed point.  Then
\(
0\notin C.
\)
By the strong separation theorem \cite[Section~11]{Rockafellar1970}, there exist \(v\in Y\setminus\{0\}\) and
\(\gamma>0\) such that
\begin{equation}\label{eq:range-separation}
	\langle c,v\rangle\ge \gamma
	\qquad \forall c\in C .
\end{equation}
Since \(G_a(b)\in C\) for every \(b\in Y\), this yields
\begin{equation}\label{eq:Ga-separation}
	\langle G_a(b),v\rangle\ge \gamma
	\qquad \forall b\in Y .
\end{equation}
For \(t>0\), set
\(
e:=\frac{v}{\|v\|}\) and 
\(
b_t:=\bar b-te .
\)
Using \(G_a=I_Y-Q_a\), we have
\[
Q_a(b_t)-\bar b
=
b_t-\bar b-G_a(b_t)
=
-te-G_a(b_t).
\]
Therefore, by \eqref{eq:Ga-separation},
\[
\begin{aligned}
	\|Q_a(b_t)-\bar b\|
	\ge
	\left\langle Q_a(b_t)-\bar b,-e\right\rangle 
	=
	t+\frac{\langle G_a(b_t),v\rangle}{\|v\|}  \ge
	t+\frac{\gamma}{\|v\|}.
\end{aligned}
\]
On the other hand, because \(\Omega\) is nonexpansive and
\(P_Y\Omega(\bar a,\bar b)=\bar b\),
\[
\begin{aligned}
	\|Q_a(b_t)-\bar b\|
	&=
	\|P_Y\Omega(a,b_t)-P_Y\Omega(\bar a,\bar b)\|  \\
	&\le
	\|\Omega(a,b_t)-\Omega(\bar a,\bar b)\|  \\
	&\le
	\|(a,b_t)-(\bar a,\bar b)\|  =
	\sqrt{\|a-\bar a\|^2+t^2}.
\end{aligned}
\]
However, for all sufficiently large \(t\), it holds that
\[
\sqrt{\|a-\bar a\|^2+t^2}
<
t+\frac{\gamma}{\|v\|},
\]
which contradicts the two preceding estimates.  Hence \(Q_a\) has a fixed
point. Thus, there exists \(b\in Y\) such that \(P_Y\Omega(a,b)=b\).  Since
\(a\in X\) was arbitrary, \(P_XK=X\).
\end{proof}

We can now complete the proof of the projection theorem.

\begin{theorem}[Projected maximal monotonicity]\label{thm:projection}
Let \(T:H\rightrightarrows H\) be maximal monotone and piecewise polyhedral,
and suppose that the slice condition \eqref{eq:slice-nonempty} holds.  Then
the projected relation \(T_*\) defined in \eqref{eq:projected} is maximal
monotone on \(H_*\) and piecewise polyhedral.
\end{theorem}

\begin{proof}
By Proposition~\ref{prop:basic}, \(T_*\) is monotone and piecewise
polyhedral.  Let
\(
  R_*:=(I_*+T_*)^{-1}\) and \(  D_*:=\dom R_*.\)
By Lemma~\ref{lem:projected-resolvent},
\[
  D_*=P_*K,
  \qquad
  K=\big\{(a_*,b_{**})\mid P_{**}R(a_*,b_{**})=b_{**}\big\},
\]
where \(R=(I+T)^{-1}\). The slice condition \eqref{eq:slice-nonempty} implies \(K\ne\emptyset\).
Indeed, if \((w_*,0)\in T(z_*,z_{**})\), then
\(
  R(z_*+w_*,z_{**})=(z_*,z_{**}),
\)
and hence \((z_*+w_*,z_{**})\in K\). Since \(T\) is maximal monotone and piecewise polyhedral, \cite[Theorem 12.12]{RockafellarWets1998} and Lemma~\ref{lem:closed-domain}
show that \(R\) is nonexpansive and piecewise affine with $\dom R = H$. Applying
Lemma~\ref{lem:global-projected-fixed-point} with
\(
  X=H_*\), \(Y=H_{**}\), and $\Omega = R$,
we obtain
\(
  D_*=P_*K=H_*.
\)
By Proposition \ref{prop:basic} and Lemma~\ref{lem:resolvent-criterion}, \(T_*\) is maximal
monotone and piecewise polyhedral.
\end{proof}

\section{Preservation Operations}\label{sec:preservation} 

The projection theorem leads to the following preservation rules.  They are
piecewise polyhedral counterparts of the classical results in \cite{Rockafellar1970Sum} and
\cite[Theorem~12.43, Corollary~12.44, and Exercise~12.46]{RockafellarWets1998}. Indeed,
the relative-interior constraint qualifications appearing there are replaced
below by nonempty intersection conditions.

\begin{theorem}
\label{thm:preservation}
Let \(M:H\rightrightarrows H\) be maximal monotone and piecewise polyhedral. The following assertions hold.
\begin{enumerate}[label=\textup{(\alph*)},leftmargin=*]
\item \textbf{Restriction of arguments.}
Let \(V\) be a subspace of \(H\), and recall the canonical
inclusion \(J_V\) defined in Section~\ref{sec:preliminaries}.  If
\[
  V\cap\dom M\ne\emptyset,
\]
then
\[
  J_V^*MJ_V:V\rightrightarrows V
\]
is maximal monotone and piecewise polyhedral.

\item \textbf{Composition with a linear map.}
Let \(L:H'\to H\) be linear, where $H'$ is a finite-dimensional real Hilbert space.  If
\[
  \operatorname{range} L \cap\dom M\ne\emptyset,
\]
then
\[
  L^*ML:H'\rightrightarrows H'
\]
is maximal monotone and piecewise polyhedral, where $L^*$ is the adjoint of $L$.

\item \textbf{Addition.}
Let \(M_1,M_2:H\rightrightarrows H\) be maximal monotone and piecewise
polyhedral.  If
\[
  \dom M_1\cap\dom M_2\ne\emptyset,
\]
then
\[
  M_1+M_2: H\rightrightarrows H
\]
is maximal monotone and piecewise polyhedral.
\end{enumerate}
\end{theorem}

\begin{proof}

For (a), by Lemma \ref{lem:pp-calculus} and \cite[Exercise 12.8]{RockafellarWets1998}, we note that $T =M^{-1}$ is maximal monotone and piecewise polyhedral. Split \(H=V\oplus V^\perp\). Since \(V\cap\dom M\ne\emptyset\), choose
\(w\in V\cap\dom M\) and \(z\in M(w)\).
Write
\(
z=(x,u_\perp)\in V\oplus V^\perp.
\)
Then
\[
(w,0)\in M^{-1}(x,u_\perp)=T(x,u_\perp),
\]
so the slice condition \eqref{eq:slice-nonempty} holds for \(T\). By Theorem~\ref{thm:projection}, the projected relation \(T_*\) is maximal monotone and piecewise polyhedral.
Note that for any \(x,w\in V\),
\[
  w\in T_*(x)
  \Longleftrightarrow
  \exists u_\perp\in V^\perp\ \text{such that}\ (w,0)\in M^{-1}(x,u_\perp)
  \Longleftrightarrow
  w\in (J_V^*MJ_V)^{-1}(x).
\]
Thus, \(J_V^*MJ_V\) is the inverse of \(T_*\), and therefore is maximal monotone and piecewise
polyhedral.

For part~\textup{(b)}, let \(V:= \operatorname{range} L\) and
\(J_V:V\hookrightarrow H\) be the canonical inclusion, and define the surjective linear mapping
\(
L_V:H'\to V
\) by $L_V x:=Lx$ for $x\in H'$.
By part~\textup{(a)}, the mapping
\(
M_V:=J_V^*MJ_V
\)
is maximal monotone and piecewise polyhedral.
The maximal monotonicity of \(M_V\) implies that \(\dom M_V\) is
nearly convex by \cite{Rockafellar1970Virtual} and \cite[Theorem~12.41]{RockafellarWets1998}, which, together with \( \dom M_V = V \cap \dom M\ne\emptyset\) and \cite[Theorem 6.2]{Rockafellar1970}, yields
\[
\emptyset \ne \ri(\dom M_V) \subset \dom M_V \subset V.
\]
Moreover, the surjectivity of \(L_V\) gives
\(\operatorname{range}L_V=V\). Hence the classical
relative-interior qualification
\(
\operatorname{range}L_V\cap\ri(\dom M_V)\ne\emptyset
\)
holds. Consequently, \(L_V^*M_VL_V\) is maximal monotone \cite[Theorem~12.43]{RockafellarWets1998}. Finally,
\(L =J_VL_V\), and therefore for all \(x\in H'\),
\[
L^*ML(x)
=L_V^*J_V^*MJ_VL_V(x)
=L_V^*M_VL_V(x).
\]
The piecewise polyhedrality of $L^*ML$ follows from
Lemma~\ref{lem:pp-calculus} and the fact that
\(
\gph(L^*ML) = \{(x,L^*u) \mid (Lx, u) \in \gph M\}.
\)

For \textup{(c)}, define
\(
\mathcal M(x_1,x_2):=M_1(x_1)\times M_2(x_2)
\) for \((x_1,x_2)\in H\times H\). Since $M_1,M_2$ are maximal monotone, we have that \(\mathcal M\) is monotone, and 
\[
(I +\mathcal M)^{-1}(x_1, x_2)=((I+M_1)^{-1}(x_1),(I+M_2)^{-1}(x_2)), \quad \forall\, (x_1,x_2)\in H\times H.
\]
Hence, $\dom(I +\mathcal M)^{-1} = H\times H$, which, by Lemma~\ref{lem:resolvent-criterion}, implies that \(\mathcal M\) is maximal monotone. It is piecewise polyhedral because of Lemma~\ref{lem:pp-calculus} and the fact that $\gph \mathcal M$ is the image of
\(\gph M_1\times\gph M_2\) under a coordinate permutation.
Now define the diagonal embedding
\(
L:H\to H\times H\) by 
\(
Lx=(x,x)
\) for $x\in H$, with the adjoint
\(
L^*(u_1,u_2)=u_1+u_2
\) for $(u_1,u_2)\in H\times H$.
Since \(
\dom M_1\cap\dom M_2\ne\emptyset,
\) we have from the definitions of $L$ and ${\cal M}$ that
\(
\operatorname{range}L \cap\dom\mathcal M\ne\emptyset .
\)
Thus,
part \textup{(b)} asserts that $L^*{\cal M}L$ is maximal monotone and piecewise polyhedral.
Finally, for every \(x\in H\),
\[
\begin{aligned}
	L^*\mathcal M L(x)
	=\{u_1+u_2\mid u_1\in M_1(x),\ u_2\in M_2(x)\}
	=(M_1+M_2)(x).
\end{aligned}
\]
Hence, \(M_1+M_2\) is maximal monotone and piecewise
polyhedral.
\end{proof}

\begin{remark}
  The linear composition result also has an immediate affine extension.  Indeed, given a  maximal monotone and piecewise polyhedral mapping $M$, a linear mapping \(L:H'\to H\) and \(h\in H\), if
\[
  (\operatorname{range} L + h) \cap\dom M\ne\emptyset,
\]
then the mapping
\(
  x\mapsto L^*M(Lx+h)
\)
is maximal monotone and piecewise polyhedral.  To see this, apply
part~\textup{(b)} to the translated mapping
\(
  M_h(z):=M(z+h),
\)
which is again maximal monotone and piecewise polyhedral.
\end{remark}

The following elementary example illustrates the flexibility of the nonempty-intersection conditions in Theorem~\ref{thm:preservation} compared to the classical relative-interior conditions.
\begin{example}
\label{ex:ri-failure}
For a nonempty closed convex set \(C\), let \(N_C\) denote its normal-cone
mapping. Set
\(
  \mathbb R_-:=(-\infty,0]\) and \(\mathbb R_+:=[0,\infty).
\)
Both mappings $N_{\mathbb R_-}$ and $N_{\mathbb R_+}$ are maximal monotone and piecewise polyhedral.
Since
\[
  \ri(\dom N_{\mathbb R_-})=(-\infty,0),\qquad
  \ri(\dom N_{\mathbb R_+})=(0,\infty), \qquad \ri(\dom N_{\mathbb R_-}) \cap \ri(\dom N_{\mathbb R_+} )=\emptyset,
\]
the classical relative-interior qualification fails.  Meanwhile, since  $\dom N_{\mathbb R_-}\cap\dom N_{\mathbb R_+}=\{0\}\ne\emptyset$, Theorem~\ref{thm:preservation}(c) ensures that $N_{\mathbb R_-}+N_{\mathbb R_+}$ is maximal monotone and piecewise polyhedral, which can also be verified directly by observing that
\[
  \qquad N_{\mathbb R_-}(x)+N_{\mathbb R_+}(x)=N_{\{0\}}(x) \qquad \forall\, x\in\mathbb R.
\]
\end{example}

\section{Concluding Remarks}\label{sec:conclusion}

In this paper, we show that, for piecewise polyhedral mappings, the relative-interior constraint qualifications in the standard maximal monotonicity preservation rules can be replaced by the corresponding nonempty-intersection assumptions. This yields a simple calculus for generating new maximal monotone piecewise polyhedral mappings from known ones. 
Starting from known maximal monotone piecewise polyhedral mappings,
the projected graph-slice construction in \eqref{eq:projected},
argument restriction, linear pullback, and addition yield mappings
in the same class whenever the corresponding nonempty-intersection
condition holds. 
Our results therefore provide useful tools for analyzing polyhedral multifunctions arising from variational inequalities, normal-cone mappings, complementarity systems, and proximal-point or primal-dual schemes.

\section*{Declarations}

\noindent{\bf Conflicts of Interest } The authors have no competing interests to declare that are relevant to the content of this article.

\end{document}